\documentclass{birkjour}
\usepackage{proof}
\usepackage{bussproofs}
\usepackage{amssymb}
\usepackage{amsmath}
\usepackage{booktabs}
\usepackage{enumerate} 
\usepackage{xcolor}
 \newtheorem{thm}{Theorem}[section]

 \newtheorem{prop}[thm]{Proposition}
 \theoremstyle{definition}
 \newtheorem{defn}[thm]{Definition}
 \theoremstyle{remark}
 \newtheorem{rem}[thm]{Remark}
 \newtheorem{ex}{Example}
 \numberwithin{equation}{section}

\begin{document}

\title[]
 {Natural Density and the Quantifier ``Most"}

\author[]{ Sel\c{c}uk Topal and Ahmet \c{C}evik}

\address{Department of Mathematics, Bitlis Eren University, 13000, Bitlis, Turkey}

\address{Gendarmerie and Coast Guard Academy, 06805, Ankara, Turkey}

\email{s.topal@beu.edu.tr}
\email{a.cevik@hotmail.com}

\thanks{}



\begin{abstract}

This paper proposes a formalization of the class of sentences quantified by \textit{most}, which is also interpreted as {\em proportion of} or {\em majority of} depending on the domain of discourse. We consider sentences of the form  ``\textit{Most A are B}", where \textit{A} and \textit{B} are plural nouns and the interpretations of $ A $ and $ B $ are infinite subsets of $ \mathbb{N} $.  There are two widely used semantics for \textit{Most A are B}:  (i) $C(A \cap B) > C(A\setminus B)  $ and (ii) $ C(A\cap B) > \dfrac{C(A)}{2} $, where $ C(X) $ denotes the cardinality of a given finite set $ X $.  Although (i) is more descriptive than (ii), it also produces a considerable amount of insensitivity for certain sets. Since the quantifier {\em most} has a solid cardinal behaviour under the interpretation {\em majority} and has a slightly more statistical behaviour under the interpretation {\em proportional of}, we consider an alternative approach in deciding quantity-related statements regarding infinite sets. For this we introduce a new semantics using {\em natural density} for sentences in which interpretations of their nouns are infinite subsets of  $ \mathbb{N} $, along with a list of the axiomatization of the concept of natural density. In other words, we take the standard definition of the semantics of \textit{most} but define it as applying to finite approximations of infinite sets computed to the limit.

\end{abstract}

\maketitle

\section{Introduction}
The fact that the cardinality of the set of natural numbers and that of the set of positive even numbers are the same is surprising for most of beginner mathematicians. It is a completely mystery to them how this equality holds when one set is a proper subset of the other. Although the intuition here may rather be difficult for them at first sight, the reason behind this equality is later explained with a rigorous mathematical treatment based on Cantor's \cite{Cantor} definition of equipollency in which the equivalency is established through a bijective function between two infinite sets. Naturally, the issue comes from the idea that counting finite sets and infinite sets have a different methodology. The method of counting finite sets cannot be applied in the infinite case. For the same reason, different methods are needed to be considered to compare sets according to their various properties regarding quantity. Moreover, in order to be able to interpret a natural language mathematically and to make its fragments calculable, it is first necessary to emphasize what these fragments are meant to describe.  Any semantics of natural languages becomes sound in the finite parts of mathematics and models as long as it serves the  purposes that we have just mentioned. If we want to interact with infinitary models of natural languages, we must revise our point of view as Cantor did with the comparison of infinite sets. In this sense we are introducing a blended semantics that flows from the finite to the infinite, taking into account the features of the quantifier \textquotedblleft\textit{most}". As Moss \cite{Apll} stated, we must not hesitate to use applied logic and applied mathematics so that the blessings of mathematics, natural language, and logic can solve each other's problems.

\section{\textquotedblleft\textit{Most}\textquotedblright~ and Its Historical Surroundings }

Aristotle's original study is the oldest formal work on logic. This reputable work is a common ground where modern logic studies are taking place in. Although Aristotle's objects belong to the real world, creation of basic rules allows us to form abstractions and for us to rely on a firm basis. The standard quantifiers (\textit{for all, for some}) became the defining footprint of the language of first order logic. These quantifiers have been regarded as the established operators since Aristotle. However, Frege \cite{Frege} established a logical and symbolic relationship between the quantifiers. Thus, Frege became the first person who studied the quantifiers in the modern sense. This symbolic treatment led to the invention of additional quantifiers such as \textit{Many, More, at least n}, and \textit{ at most n}. The reader may refer to Sher \cite{sher1,sher2} for a detailed account on historical developments and notes about semantics and quantifiers. Researchers on Aristotelian logic, particularly \L ukasiewicz, Corcoran and Smiley \cite{Luka,Corcoran,Smile}, enlarged traditional research topics to problems about the soundness and completeness of ancient logical systems. Thompson \cite{Thompson} extended the Aristotelian system by adding \textquotedblleft\textit{many, more}, and \textit{most}\textquotedblright~, staying faithful to the principles of the original system, and introduced \textit{five-quantifier square of opposition}. Briefly, syllogistic theories have been taking place in many applications of wide areas, particularly generalized quantifiers of natural language theory \cite{LSM:1,Ejik1,Ejik2,Ejik3,Robet,Square2,quanVanBetNatural},  algebraic structures \cite{Alg1,alg2,alg3,alg4}, formal logic \cite{Moss1,Moss2,Moss3,RelatiSyllo,MossAdj}, unconventional systems \cite{UnconSyll,perfor,infor}, and some applications on infinite sets \cite{Infinite:2}. 

Endrullis and Moss \cite{MossMOST} introduced a syllogistic logic containing the quantifier {\em most} on finite sets. They provided the semantics of the sentence \textit{Most A are B} in a given model as follows:  

\begin{center}
(1) \textit{Most A are B} is true if and only if  $ C([[A]]\cap [[B]]) > \dfrac{C([[A]])}{2}.$ 

\end{center}

\noindent where \textit{A} and \textit{B} are plural common nouns, $ [[A]] $ denotes the interpretation of $ A $ as a set,  and the function $ C([[A]])$ denotes the cardinality of  $[[A]]$.  The authors stated that the sentences by quantified \textit{most} have a strict majority graph property in any proof searching within this logical system.  On the other hand, Westerst\r{a}hl's  interpretation of \textit{most} and branching quantifiers \cite{wester1,wester2,sherbrach} is as follows:

\begin{center}

(2) \textit{Most A are B}  is true if and only if  $ C([[A]] \cap [[B]]) > C([[A]] \setminus [[B]]).$
\end{center}

\noindent Notice that both interpretations, (1) and (2), agree in all finite models.  We omit the generalization of the agreement as a philosophical and linguistic problem since it is out of the scope of this paper.

Zadeh and Hackl \cite{MostZadeh0,MostFuzzy1,MostFuzzy2,Hackl}, in fuzzy logic and formal semantics, respectively, considered the quantifier \textit{most} as a proportional operator (determiner). Hackl pointed out that
\vspace{0.5cm}

\textit{The proportional quantifier most, in particular, supplied the initial motivation for adopting Generalized Quantifier Theory (GQT) because its meaning is definable as a relation between sets of individuals, which are \noindent taken to be semantic primitives in GQT}.
\vspace{0.5cm}

\noindent On the other hand, Zadeh said:
 \vspace{0.5cm}
 
	  \textit{This representation, then, provides a basis for inference from premises which contain fuzzy quantifiers. For example, from the propositions \textit{Most U's are A's} and \textit{Most A's are B's}, it follows that \textit{Most$ ^2 $ U's are B's}, where Most$ ^2 $ is the fuzzy product of the fuzzy proportion most with itself}.
	  \vspace{0.5cm}

\noindent Evidentally, Zadeh and Hackl devised a similar approach despite of interpretting the quantifier \textquotedblleft \textit{most} \textquotedblright under different circumstances.

The material discussed in this section was built on finite models. We will concentrate on the behavior and the meaning of \textquotedblleft\textit{most}\textquotedblright on infinite models in the following sections.

\section{Finiteness, Infinity and Quantifiers}

Throughout the history of science and philosophy, the notion of infinity has been one of the puzzling and focal points. The existence of infinity, belonging to infinity, being a part of infinity, larger or smaller infinities, types of infinity are all still controversial issues even today. Every discipline of science has tried to find strident or flexible answers to these questions in their own way. According to Aristotle, infinity must exist as a {\em potential} rather than as actual. So infinity, in his account, is divided into two types as potential and actual. The actual infinite can be thought as an infinite set as {\em completed totality} containing infinitely many members. The {\em set} of natural numbers is a clear example to actual infinite. The potential infinite on the other hand is never ending sequence of objects without having to assume any completed totality. For any natural number there exists a greater number. Defining {\em arbitrarily many} object in this way gives rise to the notion of potential infinity. Actual infinity is a completed totality, whereas potential infinity is merely a never ending ongoing process, except that the process itself is not taken as a completed object. The reader may refer to \cite{Cooper,Lear,Hintikka} for a detailed account on this topic. 

In modern mathematics, the distinction between finite and infinite sets and the comparison between them is now accepted to be well established by Cantor's \cite{Cantor}  theory of sets. Of course, it is easy to arrive to precise and reasonable judgements when making comparisons between the sizes of finite sets. Given any two finite sets, it is quite easy to find the result of any set-theoretical operation between them, such as intersection, union, difference, cardinality comparison, ... etc. Operating with infinite sets may be a little bit more puzzling on the other hand, especially when we care about the effective computability of the operation we are doing. It is necessary to define functions, such as those described by Cantor, in order to count the elements of these sets and to compare their sizes. Formally, two sets $ A $ and $ B $ have the same cardinality, written $ |A| = |B| $, if and only if there exists a bijective function $ f : A \mapsto B $. From basic set theory we know that there exist bijections between the set ($ \mathbb{E} $) of even numbers and the set of natural numbers $\mathbb{N}$ even though ($ \mathbb{E} $) ``appears" to be of half the size of $ \mathbb{N} $. Another example is that the set $ \mathbb{K}=\{3k+1:~k\in \mathbb{N} \} $ has the same cardinality as $\mathbb{N}  $. So two sets having the same number of elements is based on Cantor's definition of equipollency and constructing a bijection between them. 

In modern sense, the standard quantification operators such as \textit{All, Some}, and \textit{No} examine, respectively, inclusion, intersection, and disjointness of sets.  These examinations do not differentiate sets whether they are of finite or infinite size and do not a give precise idea how to differentiate so. Other non-classical quantifiers such as \textquotedblleft\textit{more than}\textquotedblright and \textquotedblleft\textit{at least}\textquotedblright, that consider cardinality of sets, use \textit{aleph} definitions and already known size comparison methods on finite sets. The quantifier  \textquotedblleft\textit{most}\textquotedblright, can both determine the intersection of sets and compare their cardinalities.

A flexible way of saying \textquotedblleft \textit{strictly more than half}\textquotedblright~ without giving exact numbers, as normally done so in daily life,  is to use the quantifier \textquotedblleft \textit{most}\textquotedblright. Although it is a common practice  to use such clauses in natural languages, it is more intuitive to compare the cardinalities of the domains of nouns that are used in the clauses. For instance, rather than saying \textquotedblleft cats are strictly more than half of the number of dogs\textquotedblright , we may instead compare the cardinality of the set of cats with that of the set of dogs and then deduce whether either of them exceeds the other. It would not be necessary to perform a cardinality test between two sets if we only cared about their intersection. Conversely, there would be no need to apply the intersection operator on two sets if we merely wanted to compare their cardinalities. Since the cardinality test and the intersection test are disjoint from each other, in the sense that one does not tell any information about the other, we use the quantifier {\em most} to use these two tests together on finite sets. However, we should be more careful when using this quantifier on infinite sets. The size of $ \mathbb{K}$ is same as the size of $ \mathbb{N} $ although $ \mathbb{K}$ only contains, for every $k\in \mathbb{N}  $, $3k+1$ as elements. If we consider the set of natural numbers as a disjoint union of the sets 
\[
\{0,1,2\}, \{3,4,5\}, \{6,7,8\}, \ldots
\]
\noindent we merely find one element of $ \mathbb{K}$ in each subset above. We could say intuitively that in this case $ \mathbb{K}$ \textit{takes up space} in $ \mathbb{N} $ with the ratio of $1:3$. On the other hand, the complement of $ \mathbb{K}$, $ \mathbb{K}^{c}=\{ 0,2,3,5,6,8...\}$, takes up space in $ \mathbb{N} $ with the ratio of $2:3$. In other words, the difference (or the {\em gap} as number theorists say) between the two consecutive elements of $ \mathbb{K}$ is always 2. In this case, we would like to say that ``most" of the elements of $ \mathbb{N} $ are elements of $ \mathbb{K}^{c}  $ due to the ratio $2:3$. In fact, it suffices if more than half of the elements of a set $A$ are elements of $B$. Note however that we cannot make any inference about the density ratios between two infinite sets solely by using the notion of cardinals in a direct manner since $C (\mathbb{N} \setminus \mathbb{K}^{c})= \aleph_0 $  and ``half" of the cardinality of $ \mathbb{N} \cap \mathbb{K}^{c}   $ is $ \aleph_0 $. Furthermore, most of $ \mathbb{N}  $ is not $ \mathbb{K}^{c}  $ under  semantics (1) and (2) even though $ \mathbb{K}^{c} \subseteq \mathbb{N} $. We shall give some more examples to explain this issue in a detailed manner and we shall use the notation $ Most(A,B) $ instead of the sentence ``\textit{Most A are B}" for the abbreviation.

\begin{ex}

Let  $ \mathbb{N}^{-}$ denote  $\mathbb{N} \setminus \{1,2,3\}$.

\begin{enumerate}[(i)]
	\item $ Most( \mathbb{N},\mathbb{N}) $ is false under the interpretation (1) since $ \dfrac{C (\mathbb{N})}{2}= \aleph_0$ .

	\item $ Most( \mathbb{N},\mathbb{N}^{-}) $ is false under (1) since $ \dfrac{C (\mathbb{N})}{2}= \aleph_0 $.\\
	\item $ Most( \mathbb{N}^{-},\mathbb{N}) $ is false under (1) since $ \dfrac{C (\mathbb{N}^{-})}{2}=\aleph_0 $.\\
	\item  $ Most( \mathbb{N},\mathbb{K}) $ is false under (1) since $\dfrac{C ( \mathbb{N})}{2}=\aleph_0 $.\\
	
	\item  $ Most( \mathbb{K},\mathbb{N}) $ is false with semantics (1) since  $\dfrac{C ( \mathbb{K})}{2}=\aleph_0 $.
\end{enumerate}
\end{ex}

\noindent Notice that adding/substracting finitely many elements to/from $\mathbb{N}$ and comparing the resulting set with $\mathbb{N}$ with the {\em most} quantifier leads to wrong evaluations since division of $ \aleph_{0} $ by any finite number is equal to size $ \aleph_0 $, as can be seen from the examples given above.

\begin{ex}

Let  $ \mathbb{N}^{-}$ denote  $\mathbb{N} \setminus \{1,2,3\} $ and let $ \mathbb{K}$ denote $\{3k+1:~k\in \mathbb{N} \} $.
	\begin{enumerate}[(i)]
		\item $ Most( \mathbb{N},\mathbb{N}) $ is true under the interpretation (2) since $ C(\mathbb{N}\cap \mathbb{N})=\aleph_{0} > C (\mathbb{N}\setminus \mathbb{N})=0$.
		\item $ Most(\mathbb{N}, \mathbb{N}^{-}) $ is true under (2) since $ C(\mathbb{N}\cap \mathbb{N}^{-})=\aleph_{0} > C ( \mathbb{N}\setminus\mathbb{N}^{-})=3$.
		\item $ Most( \mathbb{N}^{-},\mathbb{N}) $ is true under (2) since $ C(\mathbb{N}^{-}\cap \mathbb{N})=\aleph_{0} > C ( \mathbb{N}^{-}\setminus\mathbb{N})=0$.
		\item  $ Most( \mathbb{N},\mathbb{K}^c) $ is false under (2) since $ C(\mathbb{N}\cap \mathbb{K}^c)=\aleph_{0} = C (\mathbb{N}\setminus \mathbb{K}^c)=\aleph_{0}$.
	\end{enumerate}
\end{ex}

\noindent All statements from (i) to (iii) in \textit{Example} 2 work well with the semantics (2). On the other hand, the same semantics does not provide us to have a credible comparison unless the set-theoretic difference between the compared sets is finite. However, we observe that semantics (2) has more advantages over the semantics (1) although (iv) in \textit{Example} 2 is supposed to be true.

\begin{prop}\normalfont \label{Lem1finite}
	If $ A $ and $ B $ are non-empty finite sets and $ A \subseteq B $, then $ Most(A,B) $  is true under the semantics (1) and (2).
\end{prop}

\begin{proof}
	Trivial.
\end{proof}

\begin{prop}\normalfont \label{Lem2infinite}
If $ A $ and $ B $ are infinite sets and $ A \subseteq B $, then $ Most(B,A) $  is not always true under  semantics (1) and (2).
\end{prop}

\begin{proof}
	It is sufficient to give a counter-example. Suppose that  $ B=\mathbb{N} $ and  $A = \mathbb{K}^{c}$. Then the equality $ C(\mathbb{K}^{c})= C(\mathbb{K})$ forms a counter-example under  semantics (2). For  semantics (1) we have $ \dfrac{C(\mathbb{K}^{c})}{2}=\aleph_0 = C(\mathbb{N}) $. 
\end{proof}

\begin{prop}

For any non-empty finite set $A$, $Most(A,A) $ is true under semantics (1) and (2).

\end{prop}

\begin{proof}
	Trivial.
\end{proof}

\begin{prop}
	
For any infinite set $A$,	$ Most(A,A) $ is false under semantics (1) but true under (2).
	
\end{prop}

\begin{proof}
	Trivial.
\end{proof}

\begin{prop}
	
	For any countable set $ B $, $ Most(A,B) $ is false under semantics (1)  if  $ A $ is a countably infinite set.
	
\end{prop}

\begin{proof}
	Trivial.
\end{proof}

\begin{prop}\label{finiteSetdifference}
	Let $ A $ and $ B $ be two non-empty infinite subsets of $ \mathbb{N} $ and let $ C (A \cap B)= \aleph_0 $. Then, $ Most(A,B) $ is true under the semantics (1) and (2) whenever $A\setminus B  $ is finite. 
\end{prop}

\begin{proof}
It is easy to see that if $A\setminus B  $ is finite, then $ \aleph_{0} > n $ for any natural number $n$.
\end{proof}

From the propositions and the examples given above, it is clear that semantics (2) gives more meaningful results than semantics (1) in many aspects, but the resulting misrepresentations and inexplicable results force us to introduce a new semantics. It may often be the case that we claim certain sentences to be false since the semantics may be mathematically meaningless for the used operations. {In this case we lose from the start. How Cantor uses the set theoretic operations and cardinality comparisons defined on finite and infinite sets, in his works, are separated each other. That is, cardinalities of two sets may still be equal to each other even though one may be a proper subset of the other set. In this aspect, we should be able to make a distinction for finite and infinite sets in the usage of ``\textit{most}" just as Cantor did a similar separation when using the set theoretic operations on sets of finite and infinite sizes.

\section{A Semantics: \textit{Natural Density}}

We understand from Cantor's interpretation that to compare the cardinality of sets, we must first define an injective or bijective function between them. The fact that functions of these types are dependent on a sequence, that is, to have a construction rule, is critical for the determining the existence of the functions. For instance, finding bijective functions from any set having arithmetic progressions to $\mathbb{N}$ is quite easy since any progression considers sequence of numbers in which the difference of any two successive members is constant. Thus, we can decide on the cardinality of these sets. $ \mathbb{K} $ is an example to an arithmetic progression and its cardinality is $ \aleph_{0} $.
As we saw that $Most(\mathbb{N}, \mathbb{K}^c) $ is false under both semantics (1) and (2). However, ratio of the space occupied by $ \mathbb{K} $ in $ \mathbb{N} $, $ \dfrac{1}{3} $, is less than the ratio of the space occupied by $ \mathbb{K}^c $ in $ \mathbb{N} $, $ \dfrac{2}{3} $ that is. So ``most" of the elements of $ \mathbb{N} $ must belong to $ \mathbb{K}^c $. That is, more than half of $ \mathbb{N} $ must be $ \mathbb{K}^c $. Any number greater then the ratio $\dfrac{1}{2}$ is sufficient for the quantifier {\em most} to hold in comparison of two domains. So the ratio $\dfrac{2}{3} $ is also a reasonable amount for our example. Furthermore, it would be more meaningful if we took the space occupied by $ \mathbb{N} $ in $ \mathbb{N} $ as 1.  However, we have  $\aleph_{0}= c(\mathbb{K}^c) > c(\mathbb{K})=\aleph_{0}  $ using Cantor's idea of cardinal comparison, and this is not compatible with the spirit of the \textquotedblleft\textit{most}'' quantifier. 

We propose a new semantics, so-called \textit{natural density} (asymptotic density), in the context of complete inadequacy of semantics (1), advantages and disadvantages of semantics (2), and the richness provided by the concept of gaps and ratios. 

Natural density \cite{AD:} is a method to measure the thickness of a subset of the natural numbers, unlike Cantor's approach. In other words, the natural density is one of the possibilities to measure how {\em thick} a subset of $\mathbb{N}$ is. Several other density operators for different purposes include logarithmic, weighted, uniform, and exponential. We assume that the number 0 does not belong to $ \mathbb{N} $ in order to stay aligned with the commonly accepted (set)-theoretic model approach, i.e. we will take $ \mathbb{N} $ as a set of positive integers. We will now give some definitions and properties of the natural density and then later we will discuss the new semantics.

\begin{defn}\label{asym}
	A set $ A $ is {\em  asymptotic } to set $ B $, written $ A \sim B $, if the symmetric difference $ A\vartriangle B$ is finite. 
\end{defn}

\begin{defn}\label{nd}
Let $ A \subseteq \mathbb{N} $ be a set and let 

\begin{center}
$ d(A)= \lim_{n \rightarrow \infty} \dfrac{|\:A \cap \{1,2,...,n\}\:|}{n}   $.
\end{center}

\noindent If the limit exists, then $d(A)$ is called the {\em lower asymptotic  (natural) density} of $ A $. We will simply call this the  natural density of $ A $ in the rest of the sections to be consistent with the title of the paper. 
\end{defn}

\noindent So natural density is a kind of \textquotedblleft measure" to attribute a thickness value to an (infinite) arithmetic sequence of natural numbers like $ d(\{k,2k,3k,4k, . . . \})= \dfrac{1}{k} $ for $k\in\mathbb{N}$. Definition \ref{nd} emphasizes that the natural density is a limit that may not exist. For this reason, we will continue assuming that each set has a density. Although this may at first seem like a weakness of the semantics we are presenting, the difficulties of generalizability of sets whose density cannot be calculated (or whose density does not exist) are also obvious to settle with. That is, the Cantorian approach has some downfalls when dealing with the cardinality calculations such sets.  Therefore, the current situation forces us to consider sets to be arithmetical. There are of course sets that are not arithmetical yet have a density value, and for sampling purposes we will not ignore these sets. Aside from our treatment of density, there also exist other valid and detailed definitions, such as lower and upper densities, which also have many applications in number theory and statistics (probability theory.  Axioms for natural density are given by the following postulates \cite{Phdaxioms}:
\vspace{0.5cm}

\noindent Let $ d: P(\mathbb{N}) \rightarrow [0,1] $ be a function and let $ A,B \in\mathbb{N} $. 
\begin{enumerate}[(1)]
	\item For all $A$, $ 0\leq d(A) \leq 1 $.
	\item $ d(\mathbb{N}) =1$ and $ d(\mathbb{\emptyset}) =0$.
	\item If $ A \sim B $, then $ d(A)=d(B) $.
	\item If  $ A \cap B=\emptyset $, then $ d(A)+d(B) \leq d(A\cup B) $.
	\item For all $A$ and $B$, $ d(A)+d(B) \leq 1+ d(A\cap B) $. 
\end{enumerate}

 \noindent The notion of asymptoticity in Definition \ref{asym} is a critical point in our study. Having this interpretation using the notion of asymptoticity, this notion is also compatible with (iii) and (iv) in \textit{Example} 2 since both are true under natural density. We will see that $ Most(\mathbb{N}, \mathbb{N}^{-}) $ and $ Most(\mathbb{N}^{-}, \mathbb{N}) $ are true under the semantics (2) and the axiom (3). \\

\paragraph{Useful properties} Some of the following properties, presented in the work of Grekos \cite{G1}, Buck \cite{generalizedasym} and Niven \cite{Iven}, will support our study.
\begin{enumerate}[(i)]

	\item  $d(A)= 1- d(A^c)  $.

	\item If $ A $ is a finite subset of $ \mathbb{N} $, $ d(A)=0$.

	\item	If $ A \subseteq B $, then $ d(A) \leq d(B) $.
\end{enumerate}

\vspace{0.5cm}

It is not the main purpose of this paper to establish a logical system, but we will mention some completeness and soundness results. We shall now build a model to introduce the language and the semantics. \\

\paragraph{Syntax} We shall use the following types of expressions to keep our language and the expressions close to the syllogistic forms as much as possible. We start with a set of variables $ A, B,\ldots$ representing plural common nouns and their complements so that $ A^{c} $ denotes \textit{non}-$ A $ for any set $A$. We let $ U$ be a name. We consider sentences of the following restricted forms: 

\begin{center}
Most $ A $ are $ U $,\\
Most $ U $ are $ B $,\\
Most $ U $ are $ U $.
\end{center}

\noindent We call this language $ \mathcal{L}(Most,~d) $.\\
\paragraph{Semantics}
We are now ready to introduce the new semantics. Let the universe $ U$ be the set of natural numbers $ \mathbb{N} $ and suppose for any variable $ A $ in the language, every $ [[A]] \subseteq U $ is an infinite set which has natural density. We assume that all sets and their complements in the universe are infinite.
Any interpretation function $d: P(\mathbb{N}) \rightarrow [0,1] $, for all subsets of $ U$, must satisfy the axioms of natural density. Note that we refer to subsets of $ U $ that have natural density. The semantics allows us to use the intersection $ [[A]] \cap [[B]] $ and  set difference $ [[A\setminus B]]= [[A]] \setminus [[B]] $ for each noun $ A $ and $ B $. This gives rise to the model $ \mathcal{M}(U,d) $ (in short, $ \mathcal{M}$). Then we define truth in a model so that at least one of $ A $ or $ B $ will be $ U $ as follows:

\begin{center}
	$ \mathcal{M} \models \textit{Most(A,B)} $ if and only if $ d([[A]]\cap [[B]]) > d([[A]]\setminus [[B]]) $.\vspace{0.3cm}
\end{center}

\noindent Recall that $ Most( \mathbb{N},\mathbb{K}^c) $ is false under the semantics (2) since $ C(\mathbb{N}\cap \mathbb{K}^c)=\aleph_{0} = C (\mathbb{N}\setminus \mathbb{K}^c)=\aleph_{0}$ as given in (iv) of \textit{Example} 2. We look for an answer, using the new semantics, to questions that the semantics (2) is unable to answer. The question is that whether or not the inequality $ d([[\mathbb{N}]]\cap [[\mathbb{K}^c]]) > d([[\mathbb{N}]]\setminus [[\mathbb{K}^c]]) $ is true under the new interpretation. This translates to the sentence \textquotedblleft Most of $ \mathbb{N} $ are $ \mathbb{K}^c $", which we already think is intuitively correct. The intersection $[[\mathbb{N}]]\cap [[\mathbb{K}^c]]  $ is equal to $[[\mathbb{K}^c]]  $. If the asymptotic density of $[[\mathbb{K}^c]]  $, that is $ d([[\mathbb{K}^c]])$, is not known, we already have the property (i) telling us $ d(A)= 1-d(A^{c}) $. Then, $d([[\mathbb{K}^c]])= 1-d([[\mathbb{K}^{c^{c}}]])=1-d([[\mathbb{K}]]) $. Therefore, it is easy to compute the asymptotic destiny of $ [[\mathbb{K}]] $ to be $ \dfrac{1}{3} $, and so $ d([[\mathbb{K}^c]])= \dfrac{2}{3} $. Simplifying the inequality, we obtain $ d([[\mathbb{K}^c]]) > d([[\mathbb{K}]]) $. Finally, we obtain $ \dfrac{2}{3}>\dfrac{1}{3}  $. We see that if $A\setminus B$ is finite, then $ \textit{Most(A,B)} $ is true under both semantics, (1) and (2). In fact, we can remove the necessity of this assumption under the new semantics since $ [[\mathbb{N}]]\setminus [[\mathbb{K}^c]] $ is not finite. We have shown the correctness of the sentence under the new interpretation that we presented, however we should repeat that this semantics does not work with finite sets. Take the set $ A=\{1,2,3\} $ for instance. Then, neither $ Most(\mathbb{U}, [[A]]) $ nor $ Most( A,\mathbb{U}) $ is true since $ d([[U]]\cap [[A]] )=d([[U]]\cap [[A]] )=0$. 

Another important point is that the new semantics also works with sets whose complements are finite as mentioned in (2) and (3) of \textit{Example} 2. Truths of $ Most( \mathbb{N},\mathbb{N}^{-}) $ and $ Most( \mathbb{N}^{-},\mathbb{N}) $ are satisfied with the new semantics. Indeed, $ d([[\mathbb{N}]]\cap[[\mathbb{N}^{-}]])=1 $ and $ d([[\mathbb{N}^{-}]]\cap[[\mathbb{N}]])=1 $, and also  $ d([[\mathbb{N}]]\setminus[[\mathbb{N}^{-}]])=0 $ and $ d([[\mathbb{N}]]^{-}\setminus[[\mathbb{N}]])=0 $.

An additional example is related to the set $\mathbb{P}$ of prime numbers. It is well-known in number theory that the ratio of the space occupied by the set of all primes in $\mathbb{N}  $ is zero.  Under the new semantics, $ Most( \mathbb{N},\mathbb{P}) $ and $ Most( \mathbb{P},\mathbb{N}) $ are false as they are supposed to be. Furthermore, the truth of \textquotedblleft\textit{most natural numbers are non-prime numbers}\textquotedblright and \textquotedblleft\textit{most  non-prime  numbers are natural numbers}\textquotedblright are correctly determined under the same semantics. The primes could possibly be an escape for further studies on sets that do not have any arithmetic progression as it was proved by Tao \cite{tao} that the primes contain arbitrarily long arithmetic progressions.

\begin{rem}
We indeed force the sets (nouns in the language) and their complements in the universe to be infinite since we mentioned that the truth value of the sentences $ Most( \mathbb{N},A) $ and $ Most(A,\mathbb{N}) $ are always false when $ A $ is finite. If we allow them, or their complements, to be finite, then this would contradict the very assumption that for any variable $A$ in the language, $[[A]]\subseteq U$ is infinite. This is due to the fact that we possibly could take $ [[A]] $ to be a finite set. This would cause the sentence belonging to a given set of sentences to be false. The system would not be sound in that case. Another way to construct the language is to allow only infinite sets in the premise and allow finite sets to be in the derived sentences. Although $ Most(\mathbb{N}^{-}, \mathbb{N}) $ and $ Most(\mathbb{N}, \mathbb{N}^{-}) $ are two favorable sentences, we did not prefer to include them in the language for these reasons discussed above. If we allow any finite set to be used in any sentence to be in the deductive closure, the logical system will not change but model constructions and the completeness proof will be affected.
\end{rem}

It is worth remind the reader some definitions. We say that $ \mathcal{M} \models \Gamma $ iff $ M \models \varphi$ for every   $\varphi \in \Gamma $. We write $ \Gamma \models \varphi  $ iff for all $ \mathcal{M}$, $  \mathcal{M} \models \varphi$ whenever $ \mathcal{M} \models \Gamma $. We read the latter as $ \Gamma $ {\em logically implies} $\varphi  $ (or $ \Gamma $  semantically implies $\varphi  $, or that $\varphi  $ is a semantic consequence of $ \Gamma $). The proof system is {\em sound} under the defined semantics if whenever $ \Gamma \vdash \varphi  $, we also have $ \Gamma \models \varphi  $. The proof system is {\em complete} if whenever $ \Gamma \models \varphi  $, we have that $ \Gamma \vdash \varphi  $.

The main semantic definition is the {\em consequence} relation.   We write $ \Gamma \vdash \varphi $ to mean that for all $\mathcal{M}$, if all sentences in $ \Gamma $ are true in $\mathcal{M}  $, then so is $ \varphi $. For the natural density function $ d $, if all sentences in $ \Gamma $ are true in $\mathcal{M}  $, then so is $ \varphi $.

\begin{center}
	\begin{figure}[h!]
	\begin{tabular}{|p{12cm}| }
			\hline
			\begin{center}
				
				$$
				\infer[(Axiom)]{Most(A,U)}
				\qquad\qquad
				\infer[(X_1)]{\varphi }{Most(A^{c},U) \:\:\:Most(A,U)}
				$$
				
				$$
				\infer[(X_2)]{\varphi }{Most(U,A^{c}) \:\:\:Most(U,A)}				
				$$

			\end{center}
			\\ \hline
		\end{tabular}
		\vspace{0.2cm}
		\caption{The logic of $ \mathcal{L}(Most,d) $}
	\end{figure}
	
\end{center}

It should be understood from Figure 1 that if $ \Gamma $ is consistent and  $ \Gamma \vdash   Most(\mathbb{N}, A) $, then $\Gamma \nvdash   Most(\mathbb{N}, A^{c})  $ (by $ (X_1) $), and if $ \Gamma $ is consistent and  $ \Gamma \vdash   Most(A,\mathbb{N}) $, then $\Gamma \nvdash   Most( A^{c},\mathbb{N})  $  (by $ (X_2) $).

The introduced logic has three rules in which the first one is taken as an axiom and the others are \textit{Ex falso quadlibet} rules. The reason why there only three postulates is because  natural density is not defined for comparisons in which $ \mathbb{N} $ is not included.  Thus, for every sentence $\varphi $ except $ Most(A,U) $ (\textit{Axiom}) and for every consistent $ \Gamma $,  $ \Gamma \vdash \varphi $ if and only if $ \varphi\in \Gamma $. Equivalently, $ \Gamma \nvdash \varphi $ if and only if $ \varphi\notin \Gamma $. $\Gamma \models \varphi  $ for all $ \varphi $ in $ \Gamma $. It is obvious to see that $ \Gamma \vdash \varphi $ if and only if $ \mathcal{M} \models \varphi $ since $ \Gamma \models \varphi $ for all $ \varphi $ in $\Gamma  $. This also implies the completeness.

\section{Conclusion}

The semantics in this paper is in the scope of natural density. We have used a ratio-wise approach to the quantifier \textit{most} for introducing asymptotic density concept which is defined on the subsets of $ \mathbb{N} $. We have shown that the proposed semantics works much better compared to the semantics relied on the cardinality comparison of set-difference which is actually based on taking half of the size of the cardinality of the given set. The new semantics works much better than the one which is just based on taking half the size of the cardinality of the given set. We see that the former semantics does not work on the sets which both itself and its complement are infinite. On the other hand, the new semantics works quite well these types of sets. We have also given an axiomatization of $ \mathcal{L}(Most,~d) $ in the last section.

It should be noted that we failed to encounter any similar logical or linguistic studies already present in the literature to make comparisons like in our approach. If we do not have  sufficiently many  logical rules but have a set of sentences, how may we compare the sentences and their semantics, within logic, to determine whether or not they are compatible with each other. As  Corcoran stated in his paper \cite{CorcoranGaps}, these  \textquotedblleft logics'' are just models. 

The results given here can be extended with syllogistic systems,  the classical boolean operations and other cardinality comparisons.

This study is restricted to sets having natural density. Nevertheless, one may as well work with other density definitions such as {\em Dirichlet, logarithmic, weighted, uniform, exponential and generalizations}, depending on the purpose of their study. 

Some may find these results to be surprising. Our hope is to encourage mathematicians, linguists, computer scientists, philosophers, and logicians to collaborate with each other to find further results related to this study that might help to solve similar problems in different domains perhaps by transforming the problem into a suitable domain.

The reader may refer to Krynicki and Mostowski \cite{General} for a general discussion ignoring natural density for higher order structures.

\subsection*{Acknowledgment}
We would like to thank Lawrence S. Moss, John Corcoran and Georges Grekos for many useful discussions and patiently answering our questions.

\end{document}